\documentclass{amsart}
\usepackage[utf8]{inputenc}
\usepackage{slashed}
\usepackage{textcomp} 
\usepackage{amsmath,amstext,amsopn,amssymb, amsfonts, amsthm}  
\usepackage{mathrsfs}
\usepackage{graphicx}
\usepackage{xcolor}
\usepackage{geometry} 
\usepackage{array} 
\usepackage{textcomp} 
\usepackage[all,cmtip]{xy} 
\usepackage{bbm}
\usepackage{varioref}
\usepackage{pdfsync}
\usepackage{enumerate}

\usepackage[pdftex,bookmarks,colorlinks,breaklinks]{hyperref}  
\definecolor{dullmagenta}{rgb}{0.4,0,0.4}   
\definecolor{darkblue}{rgb}{0,0,0.4}
\definecolor{darkgreen}{rgb}{0,0.4,0}
\hypersetup{linkcolor=darkblue,citecolor=blue,filecolor=dullmagenta,urlcolor=darkblue} 

\newtheorem{thm}{Theorem}[section]
\newtheorem{defi}[thm]{Definition}
\newtheorem{prop}[thm]{Proposition}
\newtheorem{cor}[thm]{Corollary}
\newtheorem{lemma}[thm]{Lemma}
\newtheorem{rk}[thm]{Remark}

\theoremstyle{theorem}
\newtheorem{ltheorem}{Theorem}

\newcommand{\tens}{\otimes}

\newcommand{\bC}{{\mathbb{C}}}

\newcommand{\bR}{{\mathbb{R}}}

  \newcommand{\A}{{\mathcal{A}}}
  \newcommand{\B}{{\mathcal{B}}}

  \newcommand{\E}{{\mathcal{E}}}

  \newcommand{\M}{{\mathcal{M}}}
  \newcommand{\N}{{\mathcal{N}}}

\renewcommand{\phi}{\varphi}
\newcommand{\upchi}{{\raise.35ex\hbox{\ensuremath{\chi}}}}

\renewcommand{\leq}{\leqslant}
\renewcommand{\geq}{\geqslant}

\newcommand{\fin}{\hspace*{\fill} $\square$ \vskip0.2cm}

\begin{document}

\vskip-30pt

\null

\title{On spectral gaps of Markov maps}
\date{}

\author[J.M.\ Conde-Alonso]{Jos\'e M. Conde-Alonso}
\address{\noindent Departament de Matem\`atiques, Facultat de Ci\`encies, \newline \indent Universitat Aut\`onoma de Barcelona, 08193 Barcelona, Spain}
\email{jconde@mat.uab.cat}

\author[J. Parcet]{Javier Parcet}
\address{\noindent Instituto de Ciencias Matem{\'a}ticas CSIC-UAM-UC3M-UCM, Consejo Superior de
Investigaciones Cient{\'\i}ficas \newline \indent C/ Nicol\'as Cabrera 13-15. 28049, Madrid. Spain}
\email{javier.parcet@icmat.es}

\author[\'E. Ricard]{\'Eric Ricard}
\address{\noindent Laboratoire de Math{\'e}matiques Nicolas Oresme, Universit{\'e} de Caen Normandie,14032 Caen Cedex, France}
\email{eric.ricard@unicaen.fr}


\begin{abstract}
It is shown that if a Markov map $T$ on a noncommutative probability space $\mathcal{M}$
has a spectral gap on $L_2(\mathcal{M})$, then it also has one on $L_p(\mathcal{M})$ for
$1<p<\infty$. For fixed $p$, the converse also holds if $T$ is factorizable.
Some results are also new for classical probability spaces.
\end{abstract}

\thanks{JM Conde-Alonso  was supported in part by ERC Grant 32501. J Parcet was supported in part by CSIC Grant PIE 201650E030 (Spain).}

\maketitle

\addtolength{\parskip}{+1ex}


\vskip-30pt

\null

\section*{Introduction}

Many definitions of spectral gaps have been considered  for linear operators.
  They are intersting as they often yield nice properties for
functional calculus or ergodic theory. In this note we consider contractive
linear maps $T$ on (noncommutative) $L_p$-spaces whose fixed
points are $1$-complemented by some projection $\mathsf{E}$. Then we say
that $T$ has a $L_p$-spectral gap when $\|T(1-\mathsf{E})\|<1$. Of course,
in this situation, $1$ is at a positive distance from the rest of the
spectrum of $T$. When $T$ and $\mathsf{E}$ can be considered
simultaneously on all $L_p$ ($1\leq p\leq \infty$) it is a natural question to know if 
$L_p$-spectral gaps can be interpolated.  This is precisely the
topic what we address in this note.  Our motivation comes from the
paper \cite{CMP}, where this was the key to obtain certain
interpolation results that in turn yielded Calder\'on-Zygmund
estimates in nondoubling contexts.  We focus on the particular class
of Markov maps. In other words, unital, completely positive, trace preserving maps acting on 
noncommutative probability spaces.  In the commutative situation, they
exactly correspond to the usual Markov operators.    Our main
 result reads as follows (precise definitions below):

\begin{ltheorem}
\label{theoremA}
Given any Markov map $T$ and $1 < p < \infty \hskip-1pt :$
\begin{itemize}
\item[(1)] If $T$ has an $L_2$-spectral gap, then it also has an $L_p$-spectral gap.
\item[(2)] If $T$ has an $L_p$-spectral gap and is factorizable, then it also has an $L_2$-spectral gap.
\end{itemize}
\end{ltheorem}

Another way of formulating our main result is saying that, under the
additional (and very natural in examples) condition of being
factorizable, if $T$ has an $L_p$-spectral gap for some $1<p<\infty$
then it also does for all $1<q<\infty$. When the underlying space is a
classical probability space, the assumption is automatic.
The rest of the paper is
devoted to developing the necessary machinery and definitions and to
the proof of Theorem \ref{theoremA}. We will also show an application
to interpolation theory in a context |inspired by the aforementioned
\cite{CMP}| where we have two $L_p$ spaces over the same measure space
quotiented by two different subalgebras.

In most examples $L_2$-spectral gaps are easy to determine, think for instance of
Fourier multipliers over the torus. It turns out that they also behave quite well with respect to algebraic
operations. For instance if $T$ and $S$ have an $L_2$-spectral gap,
then the tensor product map $T\otimes S$ also does. Thus, Theorem
\ref{theoremA} can be used to produce many examples of $L_p$-spectral gaps.

 To prove Theorem \ref{theoremA} one may be tempted to use an
 ultraproduct argument and Mazur maps. This probably could be done but
 would require lots of technicalities, especially in the
 noncommutative situation as ultraproducts of noncommutative
 probability spaces are not probabilty spaces any longer (one has to
 deal with type III algebras). Our approach has the advantage to give
 quantitative estimates for the spectral gaps.

 K.  Oleszkiewicz kindly informed us that $(1)$ is also proved in \cite{HMO} in the commutative situation. The proof given there also works in the noncommutative setting.


\section{Markov maps and noncommutative $L_p$ spaces}

We work in the general setting of noncommutative integration, for which a rather complete introduction and definitions can be found in \cite{PX}. Let $(\M,\tau)$ be a noncommutative probability space, so that $\mathcal{M}$ is a finite von Neumann algebra equipped with a normal faithful tracial state $\tau$. Given $1\leq p<\infty$, the noncommutative $L_p$ spaces associated to $(\M,\tau)$ are defined as 
$$
L_p(\M)= \Big\{ f\in L_0(\M,\tau)\;:\; \|f\|_p = \tau \big(|f|^p\big)^{\frac{1}{p}}<\infty\Big\}.
$$ 
Above, $L_0(\M,\tau)$ denotes the set of $\tau$-measurable operators. Strictly speaking, we should refer to the trace $\tau$ in the notation for $L_p(\mathcal{M})$, but this will not be relevant here. As usual, we can think that $\M=L_\infty(\M)$ is represented in $\mathcal{B}(L_2(\M))$, the bounded linear operators on $L_2(\mathcal{M})$, by left
multiplications. It is possible to avoid $L_0(\M,\tau)$ in the definition of the $L_p$ spaces in our situation: $L_p(\M)$ is just the completion of $\M$ in the $L_p$ norm, because the finiteness of $\tau$ yields $L_\infty(\M)\subset L_p(\M)$. In the commutative situation, $\M$ is just $L_\infty(\Omega)$ over some probability space $(\Omega,\mu)$, $\tau$ is the integration against $\mu$ and $L_p(\M)=L_p(\Omega)$ for $1<p<\infty$.

Noncommutative $L_p$ spaces share many properties of classical $L_p$, but usually inequalities for operators are more difficult to deal with. To overcome some of the difficulties that arise due to noncommutativity, the main technical tools we will be relying on are estimates on Mazur maps.  The Mazur map is the classical norm preserving map
$$
M_{p,q}: L_p(\mathcal{M})\to L_q(\mathcal{M}) \;\; \mathrm{ given} \; \mathrm{ by} \;\;  M_{p,q}(f)=f|f|^{p/q-1},
$$ 
where as usual we take $|f|^2=f^*f$, as in the definition of the noncommutative $L_p$ norm. We know from \cite{R1} that the map $M_{p,q}$ is $\min\{1,\frac p q\}$-H\"older continuous on spheres, just as in the commutative case. We are interested in working with a particular set of maps.

\begin{defi}
A map $T:\M\to \M$ is called Markov on $(\M,\tau)$ when$\hskip1pt :$
\begin{itemize}
\item[i)] $T$ is unital$\hskip1pt :$ $T(1_{\mathcal{M}})=1_{\mathcal{M}}$,
\item[ii)] $T$ is completely positive$\hskip1pt :$ $(T(x_{i,j}))\in \mathbb{M}_n(\M)_+$ for all $(x_{i,j})\in \mathbb{M}_n(\M)_+$,
\item[iii)] $T$ is trace preserving$\hskip1pt :$ $\tau\circ T=\tau$.
\end{itemize}
\end{defi}

It is then classical that $T:L_\infty(\M)\to L_\infty(\M)$ admits a unique contractive extension to $L_p(\mathcal{M})$ for $1\leq p\leq \infty$ that we will still denote by $T$. More generally, one can give a definition of a Markov map $T:\M\to \N$ between two semifinite von Neumann algebras. Let us see now some standard examples of Markov maps:

\begin{enumerate}
\item Given a classical probability space $(\Omega,\mu)$, any unital, positive and measure preserving map $T:L_\infty(\Omega)\to L_\infty (\Omega)$ is Markov. For instance, it may be given by a composition operator $T(f)(x)=f(\phi(x))$, where $\phi:\Omega\to \Omega$ is any measure preserving transformation.

\vskip3pt

\item Let $G$ be a discrete group. Its associated group von Neumann algebra 
$$
\mathcal{L}(G)=\Big\{\lambda(g):g\in G\Big\}''\subset \mathcal{B}(\ell_2(G))
$$
is the von Neumann algebra generated by the left regular representation $\lambda(g)$. It can be naturally viewed as a noncommutative probability space with the trace given by the vector state associated to $\delta_e$, where $e$ is the unit in $G$. Any normalized positive definite function $c: G\to \bC$ gives rise to a Fourier multiplier $F_c(\lambda(g))=c_g\lambda(g)$ that is a Markov map. Moreover, when $G$ is abelian with compact Pontryagin dual $\widehat G$ and normalized Haar measure $\mu$, then $L_p(\mathcal{L}(G))=L_p(\widehat G,\mu)$. Any Markov Fourier multiplier $F_c$ as above is then given by the convolution on $\widehat G$ with a probability measure $\gamma$ such that $\widehat \gamma(g)=c_g$.

\vskip3pt

\item If $\M=\mathbb{M}_n$, the family of $n \times n$ matrices equipped with its normalized trace, any Markov map is given by $T(x)= \sum_{\ell=1}^N a_\ell x a_\ell^*$ with $a_\ell \in \mathbb{M}_n$ such that  
$$
1_{\mathbb{M}_n}=\sum_{\ell=1}^N a_\ell a_\ell^*=\sum_{\ell=1}^N a_\ell^* a_\ell.
$$
For instance, if  $S((x_{i,j}))=(s_{i,j}x_{i,j})$ is a Schur multiplier, it is Markov if and only if $(s_{i,j})\geq 0$ and $s_{i,i}=1$ for all $i$.

\vskip3pt

\item Any $*$-representation $\pi: \M\to \M$ is a Markov map if and only if it is trace preserving.

\vskip3pt

\item If $\mathcal{M}$ is finite and $\N\subset \M$ is a von Neumann subalgebra, then the trace preserving conditional expectation onto $\N$, $\mathsf{E}_{\mathcal{N}}:\M\to \M$, is a Markov map.
\end{enumerate}

Given any Markov map $T:\M\to \M$, the set of its fixed points is a von Neumann sub-algebra $\N\subset \M$. We know from \cite{Ch} that this algebra is exactly the multiplicative domain of $T$. The same holds for the $L_p$ extension of $T$: the points fixed by $T$ on $L_p(\mathcal{M})$ coincide with $L_p(\N)$. Moreover, the conditional expectation $\mathsf{E}_{\mathcal{N}}$ onto $\N$ commutes with $T$. This can be seen by applying the von Neumann ergodic theorem to $T$ on $L_2(\mathcal{M})$. We now make precise the notion of $L_p$-spectral gap that we shall be using. To that end, we need to introduce the following notation: 
$$
L_p^0(\M)=\Big\{x\in L_p(\M)\; : \; \mathsf{E}_{\mathcal{N}}(x)=0\Big\}.
$$
$L_p^0(\mathcal{M})$ is complemented in $L_p(\mathcal{M})$ by $\mathrm{Id}-\mathsf{E}_{\mathcal{N}}$. The notion of $L_p$-spectral gap is then given by certain norm estimates:

\begin{defi}
We say that a Markov map $T:\M\to \M$ with fixed points algebra $\N$ has a spectral gap on $L_p(\mathcal{M})$ if
$$
c_p: =\left \|T:L_p^0(\M)\to L_p^0(\M)\right \|<1,
$$
that is, if there is a constant $c<1$ such that for any $x\in L_p(\M)$ with $\mathsf{E}_{\mathcal{N}}x=0$, we have $\|T(x)\|_p\leq c \|x\|_p$.
\end{defi}

We can now justify the fact that having an $L_p$-spectral gap with constant $c<1$ implies that $1$ has to be an isolated point of the spectrum. Indeed, suppose not and that for $\epsilon$ arbitrarily small (any $\epsilon < (1-c)$ suffices) $\lambda$ is an element of the spectrum of $T$ such that $|1-\lambda| < \epsilon$ with associated eigenvector $x$. Then we may consider the vector $z=x-\mathsf{E}_\mathcal{N} x$, which belongs  to $L_p^0(\mathcal{M})$. Since $T$ commutes with $\mathsf{E}_\mathcal{N}$,
$$
T(z) = T(x) - T(\mathsf{E}_\mathcal{N}(x)) = T(x) - \mathsf{E}_\mathcal{N}(T(x)) = \lambda (x-\mathsf{E}_\mathcal{N}x) = \lambda z,
$$
so $z$ is also an eigenvector associated to the same eigenvalue and $\|Tz\|_p = |\lambda| \|z\|_p > c\|z\|_p$, violating the $L_p$-spectral gap condition.

Due to complementation, one is tempted to try to relate spectral gaps using complex interpolation and prove Theorem \ref{theoremA} in that way. But since $\| \mathrm{Id}-\mathsf{E}_{\mathcal{N}} : \M\to \M\|=2$ in general, this only gives that $c_p\leq c_2^{2/p}2^{1-2/p}$ for $p>2$, which is usually not enough. This is why we need to employ another approach based on the use of Mazur maps mentioned above. Also, for the backwards direction of Theorem \ref{theoremA}, we will need to consider a particular type of Markov maps:

\begin{defi}
\label{factorizable}
A Markov map $T$ on $(\M,\tau)$ is factorizable if there exist a bigger finite von Neumann algebra $(\tilde \M,\tilde \tau)\supset(\M,\tau)$ and a $*$-representation $\pi :\M\to \tilde \M$ such that
$$
\forall \,x\in \M, \;\;\; \tau(x)=\tilde \tau(x)=\tilde \tau(\pi(x)) \;\;\; \mathrm{and} \;\;\;  T(x)=\mathsf{E}_{\mathcal{M}} \pi(x).
$$
\end{defi}

Theorem \ref{theoremA} then states that if $T$ is factorizable in the sense of definition \ref{factorizable} then having a spectral gap in $L_p(\mathcal{M})$ is equivalent to having a spectral gap in $L_2(\mathcal{M})$. This notion appeared in \cite{A} and turned out to be quite useful to deal with analytical problems. It follows from \cite{HM} that there are Markov maps that are not factorizable. However, most natural examples are, see \cite{R3}. If $\M=L_\infty(\Omega)$ then all Markov maps are factorizable. This corresponds to the basic construction of Markov
chains. On the other hand, a Fourier multiplier $F_c$ on a discrete group $G$ is factorizable if $c:G\to \bR$ is positive definite, and a Schur multiplier $S$ is factorizable if it is a Markov map and $(m_{i,j})\in \mathbb{M}_n(\bR)$. Finally, a product of factorizable maps is still factorizable.


 
\section{$L_2$-spectral gap implies $L_p$-spectral gap}
\label{sec2}

This section is devoted to the proof of the forward direction of Theorem \ref{theoremA}, which is contained in the next result. We keep all notations from the previous paragraphs.

\begin{thm}\label{theoA}
Assume that a Markov map $T$ on $(\M,\tau)$ has a spectral gap on $L_2(\mathcal{M})$ with constant $c_2<1$. Then $T$ also admits a spectral gap in $L_p(\mathcal{M})$ for any $1<p<\infty$ with constant $c_p=c(p,c_2)<1$. Moreover, the following estimates hold for some universal $C>0$:
$$
\lim_{c_2\to 1}\frac {1-c_p}{1-c_2} \geq C \begin{cases} p-1 & \mbox{if } \ p < 2, \\ \displaystyle \frac{1}{p}\Big(\frac {\mathrm{log}\; 2}{2p}\Big)^p & \mbox{if } \ p > 2. \end{cases}
$$
\end{thm}

The proof of Theorem \ref{theoA} requires a couple of short auxiliary lemmas.

\begin{lemma}\label{psmall}Let $1<p<2$, $x\in L_p^+(\M)$ and $T$ as above. Then
$$
\|T(x)\|_p\leq \|T(x^p)\|_1^{2/p-1}\|T(x^{p/2})\|_2^{2-2/p}.
$$
\end{lemma}

\begin{proof}
This is a standard application of complex interpolation. If $\theta=2-\frac 2 p$, then $L_p(\mathcal{M})$ is the interpolated space $(L_1(\mathcal{M}),L_2(\mathcal{M}))_\theta$. Consider the function $F(z)=T(x^{p-zp/2})$, which is holomorphic on the strip $S=\{0<{\rm Re}\,z<1\}$ and continuous on
$\overline S$. By interpolation
$$
\|T(x)\|_p=\|F(\theta)\|_p\leq \sup_{t\in \bR} \|F(it)\|_1^{1-\theta} \sup_{t\in \bR} \|F(1+it)\|_2^{\theta}.
$$
Recall the following factorization: for any $y\in L_p(\M)$ there exists a contraction $\gamma\in \M$ such that 
$$
T(y)=T(|y^*|)^{1/2}\gamma T(|y|)^{1/2}.
$$
Therefore, if $y$ is normal then by H\"older's inequality we know that $\|T(y)\|_q\leq\|T(|y|)\|_q$. We deduce that for any $t\in \bR$, $\|F(it)\|_1\leq \|T(x^p)\|_1$ and $\|F(1+it)\|_2\leq \|T(x^{p/2})\|_2$.
\end{proof}

\begin{lemma}\label{pbig}
Let $\alpha\geq 1$, $p\geq 1$ and  $T$ as above. Then for all $x\in L_{p\alpha}^+(\M) \hskip-1pt :$
$$
\|T(x^\alpha)\|_p\geq\|T(x)^\alpha\|_p.
$$ 
\end{lemma}
\begin{proof}
The fact that $p\geq 1$ ensures that all elements are well defined. By operator convexity of the map $t\mapsto t^\alpha$, the result is obvious if $\alpha\in [1,2]$ because $0\leq T(x)^\alpha\leq T(x^\alpha)$. Therefore, to conclude it suffices to note that if the lemma holds for $\alpha$, it also holds for $2\alpha$. Indeed,
$$
\|T(x)^{2\alpha}\|_p=\|T(x)^{2}\|_{\alpha p}^\alpha\leq \|T(x^{2})\|_{\alpha p}^\alpha
=\|T(x^{2})^\alpha\|_{p}\leq \|T(x^{2\alpha})\|_p.
$$
\end{proof}

\medskip

{\noindent \bf Proof of Theorem \ref{theoA}.}  Let $x\in L_p^0(\M)$ with $\|x\|_p=1$, and assume $\|T(x)\|_p= \gamma$. We will give an upper bound for $\gamma$. To do so, we first notice that we can assume $x=x^*$. This can be justified by the use of the so-called $2\times 2$ trick. Indeed, consider 
$$
\tilde x=\frac 1 {2^{1/p}}\begin{pmatrix} 0&x\\ x^*& 0\end{pmatrix}\in L_p(\mathbb{M}_2\tens \M).
$$
Then one has $\|\tilde x\|_p=1$ and $\|\mathrm{Id}_{\mathbb{M}_2}\tens T(\tilde x)\|_p= \gamma$, $\tilde x^*=\tilde
   x$, $\mathrm{Id}_{\mathbb{M}_2}\tens \mathsf{E}_{\mathcal{N}} (\tilde x)=0$ and $\mathrm{Id}_{\mathbb{M}_2}\tens T$ is still a Markov map on $\mathbb{M}_2\tens \M$. 
   
 Using that $x$ is self-adjoint, we write $x=x_+-x_-$ the decomposition of $x$ into its positive and negative parts. Without loss of generality we can assume $\|x_+\|_p^p\geq \frac 12$. Define $\gamma_\pm$ by $\|T(x_\pm)\|_p=\gamma_\pm \|x_\pm\|_p$. We next use the fact that if $a,b\geq 0$ then $\|a-b\|_p^p\leq \|a\|_p^p+\|b\|_p^p$. Applying it to $T(x_+)$ and $T(x_-)$ yields
$$
 \|T(x_+)\|_p^p + \|T(x_-)\|_p^p \geq \|T(x)\|_p^p = \gamma^p.
$$
Therefore we have $\gamma_+^p\|x_+\|_p^p+\gamma_-^p\|x_-\|_p^p\geq \gamma^p$. At this point we need to distinguish two cases according to the value of $p$.

\noindent {\bf Case $p<2$.} Lemma \ref{psmall} applied to $x_+$ gives $\| T(x_+^{p/2})\|_2^{2-2/p} \|T(x_+^p)\|_1^{2/p-1}\geq \gamma_+\|x_+\|_p$. Since $T$ is a contraction on $L_1(\mathcal{M})$, we get
$$
\| T(x_+^{p/2})\|_2\geq \gamma_+^{\frac p {2p-2}}\|x_+\|_p^{p/2}.
$$
On the other hand, by orthogonality we have
$$
\|x_+\|_p^{p}=\|x_+^{p/2}\|_2^2=\|x_+^{p/2}-
\mathsf{E}_{\mathcal{N}}(x_+^{p/2})\|_2^2+\|\mathsf{E}_{\mathcal{N}} (x_+^{p/2})\|_2^2.
$$ 
Next, we write $T(x_+^{p/2})= \big(T(x_+^{p/2})-\mathsf{E}_{\mathcal{N}} T(x_+^{p/2})\big)+\mathsf{E}_{\mathcal{N}} T(x_+^{p/2})$. Then, by orthogonality again and the $L_2$-spectral gap assumption, we get
$$
\gamma_+^{\frac {2p} {2p-2}}\|x_+\|_p^{p}\leq c_2^2\|x_+^{p/2}- \mathsf{E}_{\mathcal{N}}(x_+^{p/2})\|_2^2+ \|\mathsf{E}_{\mathcal{N}} (x_+^{p/2})\|_2^2,
$$
which yields
$$
\big(\gamma_+^{\frac {2p} {2p-2}} -c_2^2\big)\|x_+^{p/2}-
\mathsf{E}_{\mathcal{N}}(x_+^{p/2})\|_2^2\leq \big(1-\gamma_+^{\frac {2p} {2p-2}})\|\mathsf{E}_{\mathcal{N}}
(x_+^{p/2})\|_2^2\leq \big(1-\gamma_+^{\frac {2p} {2p-2}})\|x_+\|_p^p.
$$ 
We can go back now to $L_p(\mathcal{M})$ by raising to the power $2/p$, see Lemma 2.2 in \cite{R1}. This means that we get
$$
\|x_+-\mathsf{E}_{\mathcal{N}}(x_+^{p/2})^{2/p}\|_p\leq 3 \|x_+^{p/2}-\mathsf{E}_{\mathcal{N}}(x_+^{p/2})\|_2
\|x_+\|_p^{1-p/2},
$$ 
and since $\mathsf{E}_{\mathcal{N}}(x_+^{p/2})^{2/p}\in L_p(\N)$ we arrive at
$$
\|x_+-\mathsf{E}_{\mathcal{N}}(x_+)\|_p\leq 2 \|x_+-\mathsf{E}_{\mathcal{N}}(x_+^{p/2})^{2/p}\|_p\leq 6
\|x_+^{p/2}-\mathsf{E}_{\mathcal{N}}(x_+^{p/2})\|_2 \|x_+\|_p^{1-p/2}.
$$
We conclude that either $\gamma_+\leq c_2^{(2p-2)/p}$, in which case we are done, or
$$
\|x_+-\mathsf{E}_{\mathcal{N}}(x_+)\|_p\leq 6 \sqrt {\frac {1-\gamma_+^{\frac {2p}{2p-2}}}{\gamma_+^{\frac {2p} {2p-2}} -c_2^2}} \|x_+\|_p.
$$ 
Obviously, the same estimate holds for $x_-$ and $\gamma_-$. Denote $\phi(t)=\sqrt {(1-t )/(t  -c_2^2)}$. Since we know that $\mathsf{E}_{\mathcal{N}}(x)=\mathsf{E}_{\mathcal{N}}(x_+)-\mathsf{E}_{\mathcal{N}}(x_-)=0$, we also have $1\leq
\|x_+-\mathsf{E}_{\mathcal{N}}(x_+)\|_p+\|x_--\mathsf{E}_{\mathcal{N}}(x_-)\|_p$. If $\gamma_+\leq c_2^{(2p-2)/p}$, then since $\gamma^p\leq (1+\gamma_+^p)/2$, we get an upper estimate and we are done. Therefore, we can assume that $\gamma_+> c_2^{(2p-2)/p}$.

We now split again into two cases. First, if $\|x_-\|_p \leq 1/4$, then $6 \leq \|x_+-\mathsf{E}_{\mathcal{N}}(x_+)\|_p$ and therefore $\phi(\gamma_+^{\frac {2p} {2p-2}})\geq 1/12$. That means 
$$
\gamma_+\leq \Big(\frac {12^2+c_2^2}{12^2+1}\Big)^{\frac {2p-2}{2p}} \;\; \mathrm{and} \; \mathrm{hence} \;\; \gamma^p\leq \frac {1+\gamma_+^p}2 \leq \frac{1+\Big(\frac {12^2+c_2^2}{12^2+1}\Big)^{\frac {2p-2}{2p}}}{2}.
$$
Finally, if $\|x_-\|_p > 1/4$, then $\gamma^p\leq 1-(1-\gamma_-^p)/(4^p)$ and as above we can assume $\gamma_-> c_2^{(2p-2)/p}$. But then $1/6 \leq \phi(\gamma_+^{\frac {2p}{2p-2}})+\phi(\gamma_-^{\frac {2p} {2p-2}})$ and one of these two terms has to be bigger than $1/2$. Hence 
$$
\gamma_+\wedge \gamma_-\leq \Big(\frac {12^2+c_2^2}{12^2+1}\Big)^{\frac{2p-2}{2p}} , \;\;\mathrm{which} \;\mathrm{also} \;\mathrm{gives} \;\mathrm{the} \;\mathrm{bound}\;\; \gamma^p\leq 1-\frac{1-\Big(\frac {12^2+c_2^2}{12^2+1}\Big)^{\frac{2p-2}{2}}}{4^p}.
$$
This is the worst possible bound. Regarding the quantitative behavior when $c_2 \to 1$, it is easy to check that there is some $C>0$ independent of $p\in ]1,2]$ so that 
$$
\lim_{c_2\to 1} \frac {1-c_p}{1-c_2}\geq C(p-1).
$$

\noindent {\bf Case $p>2$.} We use Lemma \ref{pbig} this time to get 
$$
\| T(x_+^{p/2})\|_2\geq \gamma_+^{\frac p {2}}\|x_+\|_p^{p/2}.
$$
As in the situation when $p<2$, from the above display we derive
$$
\big(\gamma_+^{p} -c_2^2\big)\|x_+^{p/2}- \mathsf{E}_{\mathcal{N}}(x_+^{p/2})\|_2^2\leq \big(1-\gamma_+^{p} )\|x_+\|_p^p .
$$
In this case, the way to go back to $L_p(\mathcal{M})$ by raising to power $2/p$ is via Ando's inequality (see Lemma 2.2 in \cite{CPPR}): 
$$
\|x_+-\mathsf{E}_{\mathcal{N}}(x_+)\|_p\leq 2 \|x_+-\mathsf{E}_{\mathcal{N}}(x_+^{p/2})^{2/p}\|_p\leq
2\|x_+^{p/2}-\mathsf{E}_{\mathcal{N}}(x_+^{p/2})\|_2^{2/p}.
$$ 
So either $\gamma_+\leq c_2^{2/p}$ or 
$$
\|x_+-\mathsf{E}_{\mathcal{N}}(x_+)\|_p\leq 2\left({\frac {1-\gamma_+^{p}}{\gamma_+^{p}-c_2^2}} \right)^{1/p}\|x_+\|_p=:2\psi(\gamma_+^p)\|x_+\|_p.
$$ 
We discuss as before: if $\gamma_+\leq c_2^{2/p}$ then $\gamma\leq [(1+c_2^2)/2]^{1/p}$. Otherwise, by assumption we have $\|x_+\|_p \geq 1/2$ so that $\|x_-\|_p\leq 1/2^{1/p}$. But by Corollary 2.5 in \cite{R2}, for $a\geq 0$ and $p>2$, we have $\|a-\mathsf{E}_{\mathcal{N}} a\|_p\leq \|a\|_p$. This implies that
$$
1- \frac{1}{2^{\frac1p}}\leq 1 -\|x_-\|_p \leq 1- \|x_--\mathsf{E}_\mathcal{N}\|_p \leq \|x_+ -\mathsf{E}_\mathcal{N}x_+\|_p,
$$
so one gets $\delta_p:=  \frac 12 (1-\frac 1{2^{1/p}})\leq \psi(\gamma_+^p)$. This leads to 
$$
\gamma\leq \left(\frac {1 +\delta_p^p \frac{1+c_2^2}2}{1+\delta_p^p}\right)^{1/p},
$$
which is enough for our purpose. Finally, one easily checks that for some universal $C$, 
$$
\lim_{c_2\to 1}\frac {1-c_p}{1-c_2}\geq \frac{C}{p}\Big(\frac {\mathrm{log}\; 2}{2p}\Big)^p.
$$
\fin

\begin{rk}{\rm 
As pointed out in \cite{CMP}, the result above is false for $p=1,\infty$, even when $T$ is a conditional expectation.}
\end{rk}

\begin{rk}{\rm 
In Theorem \ref{theoA} the requirement that $(\mathcal{M},\tau)$ is a probability space can be relaxed. If $\tau$ is only semifinite, the conclusion of the theorem holds if one adds the additional assumption that the set of fixed points satisfies
$$
\Big\{x: \|T(x)\|_2 = \|x\|_2\Big\} = L_2(\mathcal{N})
$$ 
for some von Neumann algebra $\mathcal{N}$ that is semifinite for $\tau$. Notice that this new requirement is necessary in the general case. Indeed, consider $T_s:\mathcal{B}(\ell_2) \to \mathcal{B}(\ell_2)$ given by $x \mapsto sxs^*$, where $s$ is a unilateral shift. Then, $T_s$ is Markov and the set of fixed points is not a von Neumann algebra. }
\end{rk}

\begin{rk}{\rm 
For commutative probability spaces, Theorem \ref{theoA} has an elegant
proof in \cite{HMO}, Proposition 4.1. All the arguments there carry over to
von Neumann algebras. This provides a different estimate for $c_p$, namely
$c_p\leq (1-2^{2-p^*}(1-c_2))^{1/p^*}$, where $p^*=\max\{ p,p'\}$.
This behavior of $c_p$ is better than ours for $p>2$ and $c_2$ close to $1$.
  }
\end{rk}

\section{$L_p$-spectral gap implies $L_2$-spectral gap}

We keep the same setting as in the previous section. This time we only need one auxiliary lemma.

\begin{lemma}\label{pto2}
Let $T:\M\to \M$ be a Markov map. Then for all $y\in L_2(\M)$
$$
\| T(M_{2,p}(y))-M_{2,p}(y)\|_p\leq C \|T(y)-y\|_2^{\theta} \|y\|_2^{1-\theta},
$$
for some universal constant $C>0$ and $\theta=\frac 1 4\min\{\frac p 2,\frac 2 {p}\}$.
\end{lemma}
\begin{proof}
This is a variant of Corollary 2.4 (and Remark 2.5) in \cite{CPPR} where this is done for $p>2$. Let $1<p<2$ for the rest of the proof; this is the only case that we need to consider. We want an upper bound for $T(M_{2,p}(y))-M_{2,p}(y)$. As in section \ref{sec2}, by the $2\times 2$-trick, we can reduce to proving the result for $y=y^*\in L_2(\M)$. Again decompose $y=y_+-y_-$, so that
$$
T(M_{2,p}(y))-M_{2,p}(y) = T(y_+^{2/p}-y_-^{2/p}) - (y_+^{2/p}-y_-^{2/p}) = \left[ T(y_+^{2/p}) - y_+^{2/p}\right]  - \left[T(y_-^{2/p}) - y_-^{2/p} \right].
$$
We shall prove the desired estimate for $y_+$ and $y_-$ separately instead of working with $y$. To that end, write
$$
\|T(y_+^{2/p})-y_+^{2/p}\|_p \leq \|T(y_+^{2/p})-T(y_+)^{2/p}\|_p + \|T(y_+)^{2/p}-y_+^{2/p}\|_p =: \mathrm{I} + \mathrm{II}.
$$
We shall estimate $\mathrm{I}$ and $\mathrm{II}$ separately. By operator convexity of $t\mapsto t^\gamma$ for $\gamma\in[1,2]$ and its operator concavity for $\gamma\in ]0,1]$, we get
$$
0\leq T(y_+^{2/p})-T(y_+)^{2/p}\leq T(y_+^2)^{1/p}-T(y_+)^{2/p}.
$$
Then, by Ando's inequality we get
\begin{eqnarray*}
\mathrm{I}^p &\leq&
    \|T(y_+^2)-T(y_+)^2\|_1\\
&=& \tau(T(y_+^2)-T(y_+)^2)=\tau(y_+^2-T(y_+)^2)\\ &\leq& 
2\|y_+-T(y_+)\|_2\|y_+\|_2.
\end{eqnarray*}
On the other hand, by Lemma 2.2 in \cite{R1}, $\mathrm{II} \leq 3\|T(y_+)-y_+\|_2\|y_+\|_2^{2/p-1}$, and so 
$$
\mathrm{II} \leq C\|T(y_+)-y_+\|_2^{1/p}\|y_+\|_2^{1/p}
$$ 
for some universal $C$. Of course, the same estimates apply to $y_-$. But we know that $\|T(y_\pm)-y_\pm\|_2^2\leq 2 \| T(y)-y\|_2 \|y\|_2$ by the proof of Corollary 2.4 in \cite{CPPR}. Finally, we collect everything and we get 
\begin{eqnarray*}
\|T(M_{2,p}(y))-M_{2,p}(y)\|_p & \leq & \mathrm{I} + \mathrm{II} \\
& \leq & C\left(\| T(y)-y\|_2^{1/p} \|y\|_2^{1/p} + \| T(y)-y\|_2 \|y\|_2^{2/p-1} \right) \\
& \leq & C  \| T(y)-y\|_2^{1/2p} \|y\|_2^{3/2p}, \\
\end{eqnarray*}
for some universal $C>0$, which is enough.
\end{proof}

\begin{rk}{\rm 
The exponent $\theta$ is probably not optimal. }
\end{rk}

\begin{thm}\label{theoB}
Assume $T$ is a factorizable Markov map.
Let $1<p\neq 2<\infty$ and assume there is some constant $c_p<1$ such that 
for all $x\in L_p^0(\M)$
$$\|T(x)\|_p\leq c_p \|x\|_p.$$
Then  there exists a constant $c_2=c(p,c_p)<1$ such that for all $x\in L_2^0(\M)$
$$\|T(x)\|_2\leq c_2 \|x\|_2.$$
\end{thm}

\begin{proof}
We remind the reader that the factorizability assumption means that there is another finite von Neumann algebra $(\tilde \M,\tilde \tau)$ containing $(\M,\tau)$ with a trace preserving conditional expectation $\E$ and a trace preserving $*$-representation $\pi :\M\to \tilde \M$ so that $T(x)=\E \pi(x)$ for $x\in \M$. Let $x\in L_2(\M)$ with $\mathsf{E}_{\mathcal{N}}(x)=0$ and $\|x\|_2=1$. We want to find a lower bound for $\delta = \|x\|_2-\|T(x)\|_2$. First notice that by orthogonality $\delta^2 =\|\E\pi(x)- \pi(x)\|_2^2=\delta^2$. This also yields that for any $y\in \N$, $\pi(y)=y$ and $\E(y)=y$.

The idea is to use the hypothesis via the properties of the Mazur map. By Lemma \ref{pto2}, $\| \E M_{2,p}(\pi(x))- M_{2,p}(\pi(x))\|_p\leq \delta^\theta$. Set $\lambda: =\mathsf{E}_{\mathcal{N}}(M_{2,p}(x))$. Recalling that $\E \pi(\lambda) = \pi(\lambda) = \lambda$ because $ M_{2,p}(\pi(x))=\pi( M_{2,p}(x))$, the hypothesis yields
\begin{eqnarray*}
\| \E(\pi(M_{2,p}(x))-\lambda)\|_p & = & \| T(M_{2,p}(x))- \lambda \|_p\\
& \leq & c_p \| M_{2,p}(x)- \lambda\|_p \ = \ c_p \|\pi( M_{2,p}(x)-\lambda)\|_p \\
& \leq & c_p \left[ \| T(M_{2,p}(x))- \lambda \|_p+\|(\mathrm{Id}-\E) (\pi(M_{2,p}(x))-\lambda)\|_p \right]. \\
\end{eqnarray*}
Therefore, we get
$$
(1-c_p) \| M_{2,p}(x)- \lambda \|_p\leq  \delta^\theta.
$$
Taking now $\gamma=\min\{1,\frac p2\}$, the Mazur map $M_{p,2}$ is $\gamma$-H\"older with constant $Cp$ by the main theorem in \cite{R1} (for some universal $C$). Hence
$$
1=\|x\|_2=\|(1-\mathsf{E}_{\mathcal{N}})( x-M_{p,2}(\lambda))\|_2\leq \| x-M_{p,2}(\lambda) \|_2 \leq Cp \left( \frac {\delta^\theta}{1-c_p}\right)^\gamma.
$$
This means that for some universal $C$
$$
\delta \geq \left(\frac C p (1-c_p)\right)^{2/\theta}=1-c_2.
$$ 
\end{proof}


\section{An illustration}

Motivated by questions from interpolation theory in the paper \cite{CMP}, we give an illustration of Theorem \ref{theoremA}. Let $(\M,\tau)$ be a noncommutative probability space and assume that $\A$ and $\B$ are sub-algebras so that $\A\cap \B=\N$. Consider the Markov map $T=\mathsf{E}_\A \mathsf{E}_\B$. Our assumption yields that its fixed points algebra is exactly $\N$. On $L_p^0(\M)$, one can define a norm by 
$$
\|x\|_{\Sigma,p}=\|(1-\mathsf{E}_\A)x\|_p+\|(1-\mathsf{E}_\B)x\|_p.
$$ 
Of course one has $\|x\|_{\Sigma,p}\leq 4\|x\|_p$. One important question in \cite{CMP} was to know if they are equivalent. This is false in general:

\begin{prop}\label{eq}
Assume $1<p<\infty$, then $\|.\|_{\Sigma,p}$, $\|.\|_p$ are equivalent on $L_p^0(\M)$ if and only if $\mathsf{E}_\A \mathsf{E}_\B$ has an $L_p$-spectral gap.
\end{prop}

\begin{proof}
The direction in which we assume that $T$ has a spectral gap was done in \cite{CMP}. However, we include the argument here. Indeed, let $x\in L_p^0(\M)$, then
$$
\|\mathsf{E}_\A \mathsf{E}_\B x\|_p\leq c_p \|x\|_p\leq c_p\Big(\| x-\mathsf{E}_\A x\|_p + \| \mathsf{E}_\A(x-\mathsf{E}_\B x)\|_p+\|\mathsf{E}_\A \mathsf{E}_\B x\|_p\Big).
$$  
We deduce $\|\mathsf{E}_\A \mathsf{E}_\B x\|_p\leq\frac {c_p}{1-c_p} \|x\|_{\Sigma,p}$. Since
$$
\|x\|_p\leq \|x-\mathsf{E}_\A x\|_p+ \|\mathsf{E}_\A(x-\mathsf{E}_\B x)\|_p+\|\mathsf{E}_\A \mathsf{E}_\B x\|_p,
$$
we get $\|x\|_p\leq \frac {1}{1-c_p} \|x\|_{\Sigma,p}$.

Assume now $T$ has no spectral gap. Then there exists a sequence of norm one elements $(x_n)$ in $L_p^0(\M)$ such that $\|\mathsf{E}_\A \mathsf{E}_\B x_n\|_p\to 1$. Necessarily $\|\mathsf{E}_\B x_n\|_p\to 1$, and thus by the uniform convexity of $L_p(\mathcal{M})$, we must have $\|\mathsf{E}_\B x_n -x_n\|_p\to 0$. Similarly we know that $\|\mathsf{E}_\A \mathsf{E}_\B x_n- \mathsf{E}_\B x_n\|_p \to 0$. This implies that $\|x\|_{\Sigma,p}\to 0$ due to the fact that $\|\mathsf{E}_\A x_n -x_n\|_p\to 0$, which in turn follows from the above and the decomposition
$$
\mathsf{E}_\A x_n -x_n =  \mathsf{E}_\A (x_n -\mathsf{E}_\B x_n) +  (\mathsf{E}_\A \mathsf{E}_\B x_n- \mathsf{E}_\B x_n ) +(\mathsf{E}_\B x_n -x_n).
$$ 
\end{proof}

\begin{cor}
The following are equivalent 
\begin{itemize}
\item[(1)] For some $1<p<\infty$, $\|.\|_{\Sigma,p}$, $\|.\|_p$ are equivalent on $L_p^0(\M)$.
\item[(2)]  For all $1<p<\infty$, $\|.\|_{\Sigma,p}$, $\|.\|_p$ are equivalent on $L_p^0(\M)$.
\item[(3)] For some $1<p<\infty$, $\mathsf{E}_\A \mathsf{E}_\B$ has a $L_p$-spectral gap.
\item[(4)] For all $1<p<\infty$, $\mathsf{E}_\A \mathsf{E}_\B$ has a $L_p$-spectral gap.
\end{itemize}

 If this holds then the norms $(\|.\|_{\Sigma,p})_{1<p<\infty}$ on $L_\infty^0(\M)$ form an interpolation chain.
\end{cor}

\begin{proof}
Note that $\mathsf{E}_\A \mathsf{E}_\B$ is factorizable, so this is an easy combination of Proposition \ref{eq} and Theorems \ref{theoA} and \ref{theoB}.
\end{proof}

\begin{rk}
{\rm It follows from the symmetry  that if $\mathsf{E}_\A \mathsf{E}_\B$ has an $L_p$-spectral gap, 
$\mathsf{E}_\B \mathsf{E}_\A$ also does.} 
\end{rk}
\bibliographystyle{plain}

\end{document}